\documentclass[11pt,reqno]{amsart}
\usepackage{amssymb,mathtools,calc,verbatim,enumitem,tikz,url,hyperref,mathrsfs,cite,fullpage}
\usepackage{bbm}
\usepackage{stmaryrd}
\usepackage{textcomp}
\usepackage{setspace}
\usepackage{amssymb}
\usepackage{amsthm}
\usepackage{amsmath}
\usepackage{graphicx}
\usepackage{marvosym}
\usepackage{empheq}
\usepackage{latexsym}
\usepackage{fontenc}
\usepackage{color}
\usepackage{hyperref}
\usepackage{cleveref}
\usepackage{dsfont}

\usepackage{todonotes}

\newenvironment{poc}{\begin{proof}[Proof of claim]}{\end{proof}}

\newtheorem{theorem}{Theorem}[section]
\newtheorem{lemma}[theorem]{Lemma}

\newtheorem{proposition}[theorem]{Proposition}
\newtheorem{claim}[theorem]{Claim}
\newtheorem*{claim*}{Claim}

\theoremstyle{definition}

\newtheorem*{qu*}{Question}
\theoremstyle{remark}
\newtheorem{remark}[theorem]{Remark}

\renewcommand\Pr{\operatorname{\mathbb{P}}}

\renewcommand\le{\leqslant}
\renewcommand\ge{\geqslant}
\renewcommand\to{\rightarrow}

\def\<{\langle }
\def\>{\rangle }

\begin{document}

\title{Sharp threshold for Hamilton cycles in randomly perturbed sparse graphs}
\author{
Guorui Ma
\and 
Zhifei Yan}

\address{Shanghai Institute for Mathematics and Interdisciplinary Sciences (SIMIS), Shanghai, 200433, China, Research institute of Intelligent Complex Systems, Fudan University, Shanghai, 200433, China}\email{mgr18@tsinghua.org.cn}

\address{ECOPRO, Institute for Basic Science, 55 Expo-ro, Yuseong-gu, Daejeon, 34126, Korea}\email{zhifeiyan@ibs.re.kr}

\thanks{}

\begin{abstract}
We determine the sharp threshold for Hamilton cycles in randomly perturbed sparse graphs. For any $\alpha=\alpha(n)=o(1)$, let $G_{\alpha}$ be an $n$-vertex graph with minimum degree $\delta(G_{\alpha})\ge\alpha n$. We prove that if $$p\ge(1+\varepsilon)\frac{\log(1/\alpha)}{n},$$
then the union $G_{\alpha}\cup G(n,p)$ is Hamiltonian asymptotically almost surely. This significantly strengthens a recent result of Hahn-Klimroth, Maesaka, Mogge, Mohr, and Parczyk by improving the leading constant from 6 to the optimal value of 1. Crucially, we show that this bound on $p$ is best possible when $\alpha n\rightarrow\infty$, thereby establishing the exact probability threshold for Hamiltonicity in this sparse regime. Our proof relies on a robust random expansion lemma, P\'{o}sa's booster lemma, and a sprinkling argument.
\end{abstract}
	
%\tableofcontents
\maketitle

\section{Introduction}

A central theme in extremal graph theory is finding sufficient conditions for a graph to contain specific spanning structures, such as perfect matchings or Hamilton cycles. A classic result of this type is Dirac's theorem, which states that any $n$-vertex graph with minimum degree at least $n/2$ contains a Hamilton cycle. On the other end of the spectrum, in the purely random setting, a celebrated result of Korshunov, and independently P\'osa \cite{posa}, established that the binomial random graph $G(n,p)$ asymptotically almost surely (a.a.s.) contains a Hamilton cycle when
$$p = \frac{\log n + \log \log n + \omega(1)}{n}.$$

To bridge the gap between the purely deterministic and purely random settings, Bohman, Frieze, and Martin \cite{BFM} introduced the model of \emph{randomly perturbed graphs}. In this framework, one starts with a dense deterministic graph $G_\alpha$ satisfying a minimum degree condition $\delta(G_\alpha) \ge \alpha n$ for some constant $\alpha > 0$, and adds random edges from $G(n,p)$. In their seminal work, they demonstrated that adding merely a linear number of random edges is sufficient: $G_\alpha \cup G(n,p)$ is Hamiltonian a.a.s.\ provided 
$$p \ge C/n$$
for some constant $C$ depending only on $\alpha$. This $O(1/n)$ probability threshold is a significant reduction from the $O(\log n / n)$ required in the purely random setting, powerfully illustrating how a small random perturbation can fundamentally change the threshold for the appearance of global structures. Recently, Espuny D\'iaz and Razafindravola \cite{ER} further refined the threshold behaviour in this dense regime by conjecturing the optimal constant $C$ as a function of $\alpha$; crucially, they also established the exact random edge threshold when the deterministic minimum degree is close to Dirac's condition.

While the $O(1/n)$ threshold behaviour is firmly established for constant $\alpha$, determining the (sharp) threshold for Hamilton cycles in \emph{sparse} perturbed graphs, specifically when the density parameter $\alpha = o(1)$, has remained a compelling open problem. Hahn-Klimroth, Maesaka, Mogge, Mohr, and Parczyk \cite{HMMMP} recently made significant progress in this setting. They proved that for an $n$-vertex graph $G_\alpha$ with $\delta(G_\alpha) \ge \alpha n$, the perturbed graph $G_\alpha \cup G(n, p)$ contains a Hamilton cycle a.a.s.\ provided that 
$$p \ge \frac{(6+o(1))\log(1/\alpha)}{n}.$$  
Furthermore, they observed that the highly imbalanced complete bipartite graph $K_{\alpha n, (1-\alpha)n}$ provides a natural lower bound obstruction: any spanning cycle must use enough random edges to cover the isolated vertices in the large partition class.

Our main contribution is to determine the exact probabilistic threshold of this model by tightening the leading constant to its optimal first-order value. By applying the robust expansion properties of random graphs, combined with P\'osa's rotation techniques and a multi-stage sprinkling argument, we improve the leading constant from $6$ down to $1$. 

\begin{theorem}\label{thm:main}
Fix $\varepsilon>0$ and let $\alpha=\alpha(n)=o(1)$. Let $G_\alpha$ be any graph on $n$ vertices satisfying $\delta(G_\alpha)\ge\alpha n$. If
$$p\ge(1+\varepsilon)\frac{\log(1/\alpha)}{n},$$
then
$$\Pr(G_\alpha\cup G(n,p) \text{ is Hamiltonian}) \to 1.$$
\end{theorem}

The leading constant $1$ in Theorem~\ref{thm:main} is optimal in the regime where $\alpha n \to \infty$, exactly matching the bipartite obstruction (see Section~\ref{sec:sharpness} for the formal proposition). 

This improvement is of fundamental structural significance. By leaving no slack, it proves that the threshold is precisely governed by the bipartite obstruction $K_{\alpha n, (1-\alpha)n}$, establishing the exact point at which random perturbations overcome the deterministic bottleneck.

Beyond Hamilton cycles, research on randomly perturbed graphs has expanded to various spanning and extremal properties. For instance, the threshold for embedding bounded-degree spanning trees was investigated by Krivelevich, Kwan, and Sudakov \cite{KKS}, and fully resolved by Joos and Kim \cite{JK}. This was subsequently generalized to all spanning bounded-degree graphs by B\"ottcher, Montgomery, Parczyk, and Person \cite{BMPP}. There has also been extensive work on $F$-factors (tilings). Notably, Balogh, Treglown, and Wagner \cite{BTW} and Han, Morris, and Treglown \cite{HMT} explored the gap between the Hajnal-Szemer'edi theorem and the Johansson-Kahn-Vu theorem; very recently, this line of research culminated in the work of Antoniuk, Kam\v{c}ev, Reiher, and Tukara \cite{AKRT}, who established the complete exact threshold picture for clique factors. In a distinct but related direction of near-spanning structures, Cheng, Liu, Wang, and Yan \cite{CLWY} determined the sharp threshold for near-perfect triangle packings. The model has also been fruitfully applied to Ramsey properties, as demonstrated by Das and Treglown \cite{DT}.

\subsection*{Proof Outline and Organization}
The proof relies on splitting the random edges into two independent rounds. In Section~\ref{sec:tools}, we introduce the standard tools of P\'osa boosters and a sprinkling lemma. In Section~\ref{sec:expansion}, we prove a uniform random expansion lemma (Lemma~\ref{lem:expansion}) for $G(n,p)$. In Section~\ref{sec:proof}, we expose the first round of random edges and use the expansion lemma to show that the resulting union with $G_\alpha$ forms a robust, connected 2-expander. This structural property guarantees a large reservoir of P\'osa boosters. We then expose a second, much smaller round of random edges (sprinkling) to hit these boosters and close a Hamilton cycle. Finally, in Section~\ref{sec:sharpness}, we present the sharpness construction.

\subsection*{Notation}
Throughout the paper, all logarithms are natural. We write ``a.a.s.'' for ``asymptotically almost surely''. For a graph $G$ and a set $X\subseteq V(G)$, we write $N_{G}(X)\setminus X$ for the external neighbourhood of $X$ in $G$.
\medskip

\section{Two standard tools}\label{sec:tools}

In this section, we record two standard structural tools based on P\'osa's elementary rotation technique. First we introduce a classical P\'osa boosters lemma as follows.

\begin{lemma}[P\'osa boosters]\label{lem:posa}
Let $G$ be a connected non-Hamiltonian graph on $n$ vertices. Suppose that for some integer $k$,
$$|N_{G}(X)\setminus X|\ge 2|X| \quad \text{for every } X\subseteq V(G), \ |X|\le k.$$
Then $G$ has at least $(k+1)^{2}/2$ boosters, where a booster is a non-edge $e$ such that $G+e$ is Hamiltonian or has a path longer than a longest path of $G$.
\end{lemma}
\begin{proof}
This is the usual P\'osa rotation argument. Let $P=v_{0}v_{1}\dots v_{l}$ be a longest path of $G$. For a fixed endpoint $v_{0}$, let $R(v_{0})$ be the set of endpoints obtainable from $P$ by a sequence of rotations that keep $v_{0}$ fixed. P\'osa's elementary observation gives
$$|N_{G}(R(v_{0}))\setminus R(v_{0})|\le 2|R(v_{0})|-1.$$
If $|R(v_{0})|\le k$, this contradicts the assumed expansion. Hence $|R(v_{0})|\ge k+1$. For each $x\in R(v_{0})$, repeat the same argument with $x$ fixed. We obtain a set $R(x)$ of at least $k+1$ possible opposite endpoints. For every $y\in R(x)$, the edge $xy$, if absent, is a booster: adding $xy$ closes a cycle through all vertices of the corresponding longest path, and since $G$ is connected and non-Hamiltonian, this either is a Hamilton cycle or can be extended to a longer path. Counting the resulting unordered pairs gives at least $(k+1)^{2}/2$ boosters.
\end{proof}

Another important tool is that for a connected $2$-expander, it suffices to add $\omega(n)$ random edges to obtain a Hamilton cycle a.a.s. as below.

\begin{lemma}[Sprinkling boosters]\label{lem:sprinkling}
Let $G$ be a graph on $n$ vertices satisfying
$$|N_{G}(X)\setminus X|\ge 2|X| \quad \text{for every } X\subseteq V(G), \ |X|\le n/4,$$
and suppose that $G$ is connected. Let $R\sim G(n,\lambda/n)$, where $\lambda=\lambda(n)\rightarrow\infty$ and $\lambda=O(\log n)$. Then $G\cup R$ is Hamiltonian a.a.s.
\end{lemma}
\begin{proof}
It is enough to expose $m=\lfloor\frac{\lambda n}{4}\rfloor$ random edges chosen uniformly without replacement from the complete graph, because $G(n,\lambda/n)$ contains such a random $m$-edge graph a.a.s. Let $G_{i}$ be the graph after the first $i$ exposed edges have been added to $G$. As long as $G_{i}$ is not Hamiltonian, it is still connected and still satisfies the same expansion condition. 

By Lemma~\ref{lem:posa}, with $k=\lfloor n/4\rfloor$, it has at least
$$\frac{(k+1)^{2}}{2}\ge\frac{n^{2}}{40}$$
boosters for all large $n$. Since $m=O(n \log n)=o(n^{2})$, at each non-Hamiltonian step the next exposed random edge is a booster with conditional probability at least a positive absolute constant, say $1/30$. Thus the number of successful booster steps stochastically dominates a binomial random variable with parameters $m$ and $1/30$. Its expectation is $\Omega(\lambda n)=\omega(n)$, so by Chernoff's inequality it is at least $n$ a.a.s. Each successful booster either creates a Hamilton cycle or increases the length of a longest path. Since the length of a longest path can increase at most $n$ times, $G\cup R$ is Hamiltonian a.a.s.
\end{proof}

\medskip

\section{The random expansion lemma}\label{sec:expansion}

We establish a uniform expansion property for the random graph $G(n,p)$, which serves as the core technical engine to ensure the required 2-expansion for P\'osa's lemma.

\begin{lemma}[Uniform random expansion]\label{lem:expansion}
Fix $\eta>0$. Let $\alpha=o(1)$, $L=\log(1/\alpha)$, $d=\lceil\alpha n\rceil$, and $\lambda=(1+\eta)L$. Assume $d\gg \log n$. There is a constant $K=K(\eta)>0$ such that $R\sim G(n,\lambda/n)$ has the following three properties a.a.s.
\begin{enumerate}[label=\textnormal{(E\arabic*)}]
    \item \label{E1} For every set $X$ with $\frac{Kd}{\lambda}\le|X|\le\frac{Kn}{2\lambda}$, we have $|N_{R}(X)\setminus X|>\frac{\lambda|X|}{K}$.
    \item \label{E2} For every set $X$ with $\frac{Kn}{2\lambda}<|X|\le\frac{n}{4}$, we have $|N_{R}(X)\setminus X|>\frac{n}{2}$.
    \item \label{E3} Every two disjoint sets $A, B\subseteq V$ with $|A|,|B|\ge n/4$ have at least one $R$-edge between them.
\end{enumerate}
\end{lemma}

\begin{proof}
We choose $K$ sufficiently large in terms of $\eta$. Let $X$ be fixed with $|X|=x$, and put $Z_{X}=|N_{R}(X)\setminus X|$. Then $Z_{X}\sim \text{Bin}(n-x,q_{x})$ where $q_{x}=1-(1-\lambda/n)^{x}$.

\begin{claim}\label{claim:e1}
Property \ref{E1} holds a.a.s.
\end{claim}
\begin{poc}
Write $a=\frac{\lambda x}{n}$. Then $K\alpha \le a\le K/2$. Choose a small constant $a_{0}>0$. If $a\le a_{0}$, then $\mathbb{E}Z_{X}=(1+O(a_{0})+o(1))an,$ and the standard binomial lower-tail estimate gives
$$\Pr\left(Z_{X}\le\frac{an}{K}\right)\le \exp\left[-an\left(1-O(a_{0})-\frac{1+\log K}{K}+o(1)\right)\right].$$
On the other hand,
$$\log\binom{n}{x}\le x \log\frac{en}{x}=\frac{an}{\lambda}\log\frac{e\lambda}{a}.$$
Since $a\ge K\alpha$ and $\lambda=(1+\eta)L,$ we have
$$\frac{1}{\lambda}\log\frac{e\lambda}{a}\le\frac{1+o(1)}{1+\eta}.$$
Choosing first $a_{0}$ sufficiently small and then $K$ sufficiently large, the union-bound exponent is at most $-can$ for some $c=c(\eta)>0$. Since $an=\lambda x\ge Kd$ and $d\gg \log n$, the total contribution from the subrange $a\le a_{0}$ is $o(1)$.

If $a\ge a_{0}$, then $\mathbb{E}Z_{X}=\Omega(n)$, while the threshold $an/K$ is a fixed smaller fraction of $\mathbb{E}Z_{X}$ once $K$ is chosen large. Thus $\Pr(Z_{X}\le an/K)\le \exp(-\Omega(n))$. The number of possible sets $X$ in the present range is $\exp(o(n))$, because $x=O(n/\lambda)$ and $\lambda\rightarrow\infty$. Hence this subrange also contributes $o(1)$. This proves \ref{E1}.
\end{poc}

\begin{claim}\label{claim:e2}
Property \ref{E2} holds a.a.s.
\end{claim}
\begin{poc}
If $Z_{X}\le n/2$ and $x\le n/4,$ then at least $n/4$ vertices outside $X$ have no $R$-neighbour in $X$. Hence, for fixed $X$,
$$\Pr(Z_{X}\le n/2)\le\binom{n}{n/4}\exp\left(-\frac{\lambda x}{4}\right).$$
Taking a union bound over all $X$ with $Kn/(2\lambda)<x\le n/4$ gives at most
$$n \exp\left(2H(1/4)n-\frac{Kn}{8}\right),$$
where $H(p)=-p\log p-(1-p)\log(1-p)$ is the binary entropy function. Taking $K>16H(1/4)$ proves \ref{E2}.
\end{poc}

\begin{claim}\label{claim:e3}
Property \ref{E3} holds a.a.s.
\end{claim}
\begin{poc}
For fixed disjoint $A, B$ with $|A|,|B|\ge n/4$, the probability of having no edges is
$$\Pr(e_{R}(A,B)=0)\le \exp(-\lambda n/16).$$
There are at most $4^{n}$ ordered choices of such pairs, and $\lambda\rightarrow\infty$. Thus the union bound proves \ref{E3} as required.
\end{poc}

The three claims complete the proof of the lemma.
\end{proof}

\medskip

\section{Proof of Theorem \ref{thm:main}}\label{sec:proof}

Now we are ready to establish Theorem \ref{thm:main} by a two-round sprinkling.

\begin{proof}[Proof of Theorem \ref{thm:main}]
Set $L=\log(1/\alpha)$. By monotonicity, it is enough to prove the result for $p_{0}=(1+\varepsilon)\frac{L}{n}$. Put $\lambda_{0}=(1+\varepsilon)L.$ If $\lambda_{0}\ge \log n+2 \log \log n,$ then $G(n,p_{0})$ alone is Hamiltonian a.a.s.\ by the classical Hamiltonicity threshold for the binomial random graph. Hence we may assume
$$\lambda_{0}<\log n+2 \log \log n.$$
Then $d := \lceil\alpha n\rceil \ge n \exp(-L) \ge n^{\varepsilon/(1+\varepsilon)}(\log n)^{-2/(1+\varepsilon)},$ and in particular $d\gg \log n.$

Now split the random graph into two independent rounds. Let
$$\lambda_{1}=(1+\varepsilon/2)L, \quad \lambda_{2}=\frac{\varepsilon L}{4},$$
and let $R_{1}\sim G(n,\lambda_{1}/n)$, $R_{2}\sim G(n,\lambda_{2}/n)$ be independent. Since
$$\frac{\lambda_{1}}{n}+\frac{\lambda_{2}}{n}\le(1+3\varepsilon/4)\frac{L}{n}<p_{0},$$
we may couple $R_{1}\cup R_{2}$ as a subgraph of $G(n,p_{0})$. It is therefore enough to show that $H:=G_\alpha\cup R_{1}\cup R_{2}$ is Hamiltonian a.a.s.

Apply Lemma~\ref{lem:expansion} to $R_{1}$ with $\eta=\varepsilon/2$, and fix an outcome of $R_{1}$ satisfying \ref{E1}--\ref{E3}. We claim that $H_{1}:=G_\alpha\cup R_{1}$ is a connected 2-expander up to size $n/4;$ that is,
\begin{claim}
$H_{1}:=G_\alpha\cup R_{1}$ is connected and satisfies
 $$|N_{H_{1}}(X)\setminus X|\ge 2|X| \quad \text{for all } X\subseteq V, \ 1\le|X|\le n/4.$$ 
\end{claim}

\begin{poc}
Indeed, if $1\le|X|\le d/3$, then any vertex of $X$ has at least $d$ neighbours in $G_\alpha$, at most $|X|-1$ of which lie in $X$. Therefore
$$|N_{H_{1}}(X)\setminus X|\ge d-|X|+1\ge 2|X|.$$

For larger $X$, note that $\lambda_{1}\rightarrow\infty$ so for large $n$ we have $Kd/\lambda_{1}<d/3$. Thus $|X|>d/3$ implies $|X|\ge Kd/\lambda_{1}$. If $|X|\le Kn/(2\lambda_{1})$, then \ref{E1} gives
$$|N_{H_{1}}(X)\setminus X|\ge|N_{R_{1}}(X)\setminus X|>\frac{\lambda_{1}|X|}{K}\ge 2|X|$$
again for large $n$. If $Kn/(2\lambda_{1})<|X|\le n/4$ then \ref{E2} gives
$$|N_{H_{1}}(X)\setminus X|\ge|N_{R_{1}}(X)\setminus X|>n/2\ge 2|X|.$$
So $H_{1}$ has the required expansion.

It remains to check connectedness. The expansion just proved rules out any component of size at most $n/4$. If $H_{1}$ were disconnected, then some cut would separate two vertex sets each of size at least $n/4$. This contradicts \ref{E3}, because any $R_{1}$-edge between those two sets is also an $H_{1}$-edge. Thus $H_{1}$ is connected.
\end{poc}

Finally, since $\lambda_{2}=\frac{\varepsilon L}{4}\rightarrow\infty$ and $\lambda_{2}=O(\log n)$ under the present assumption $\lambda_{0}<\log n+2 \log \log n$, Lemma~\ref{lem:sprinkling} applied to $G=H_{1}$ and $R=R_{2}$ shows that $H_{1}\cup R_{2}=G_\alpha\cup R_{1}\cup R_{2}$ is Hamiltonian a.a.s. This proves the theorem.
\end{proof}

\medskip

\section{Sharpness of the constant}\label{sec:sharpness}
The leading constant $1$ cannot be improved in the main range where the deterministic minimum degree is genuinely growing, for instance when $\alpha n\rightarrow\infty$.

\begin{proposition}[Bipartite obstruction]
Assume $\alpha=o(1)$ and $\alpha n\rightarrow\infty$. For every fixed $\eta>0$, there is an $n$-vertex graph $G_\alpha$ with $\delta(G_\alpha)\ge\alpha n$ such that, if
$$p=(1-\eta)\frac{\log(1/\alpha)}{n},$$
then $G_\alpha\cup G(n,p)$ is not Hamiltonian a.a.s.
\end{proposition}

\begin{proof}
Let $A, B$ be a partition of the vertex set with $|A|=\lceil\alpha n\rceil$ and $|B|=n-|A|$, and let $G_\alpha=K_{A,B}$. Then $\delta(G_\alpha)\ge\alpha n$. Let $R\sim G(n,p)$, and let $Y$ be the number of vertices of $B$ that have no $R$-neighbour inside $B$. Then
$$\mathbb{E}Y=|B|(1-p)^{|B|-1}=(1+o(1))n\alpha^{(1-\eta)(1-o(1))}.$$
The assumption $n\alpha\rightarrow\infty$ implies $\mathbb{E}Y\rightarrow\infty$, and a standard second-moment calculation for isolated vertices in a binomial random graph gives $Y=(1+o(1))\mathbb{E}Y$ a.a.s. Moreover, since $\alpha=o(1)$,
$$\frac{\mathbb{E}Y}{|A|}=\alpha^{-\eta+o(1)}\rightarrow\infty.$$
Hence $Y>|A|$ a.a.s. 

Every vertex counted by $Y$ has no edge to $B$ in $G_\alpha\cup R;$ in any Hamilton cycle both of its two cycle-neighbours would therefore have to lie in $A$. Since the vertices of $A$ contribute only $2|A|$ incident cycle-edges in total, at most $|A|$ such vertices can be accommodated. This contradicts $Y>|A|$. Therefore $G_\alpha\cup R$ is not Hamiltonian a.a.s.
\end{proof}

\begin{remark}
If $\alpha n$ is bounded, the usual random-graph threshold around $(\log n)/n$ may dominate the lower-bound example above. The theorem remains true because the case $\lambda\ge \log n+2\log\log n$ is already handled by the random graph alone.
\end{remark}

\medskip

\section*{Acknowledgement}

We sincerely thank Alberto Espuny D\'iaz for pointing out that the sharp threshold for Hamilton cycles in the randomly perturbed dense case (where $\alpha > 0$ is constant) remains an open problem, except for graphs very close to Dirac's condition. We also thank Jie Han and Hong Liu for helpful discussions. 

Z. Yan was supported by the Institute for Basic Science (IBS-R029-C4).

%%%%%%%%%%%%%%%%%%%%%%%%%%%%%%%%%%%%%%%%%%%%%%%%%%%%%%%%%%%%%%%%%%%%%%%%%%%%%%%%%%%%%%%%%%%%%%%%%%%%%%%%%

\medskip

\newcommand{\etalchar}[1]{$^{#1}$}


\begin{thebibliography}{HKMM{\etalchar{+}}21}

\bibitem[AKRT26]{AKRT}
S.~Antoniuk, N.~Kam\v{c}ev, C.~Reiher, and T.~P. Tukara.
\newblock The complete picture for clique factors in randomly perturbed graphs.
\newblock {\em arxiv:2603.22081}, 2026.

\bibitem[BFM03]{BFM}
T.~Bohman, A.~Frieze, and R.~Martin.
\newblock How many random edges make a dense graph {H}amiltonian?
\newblock {\em Random Structures Algorithms}, 22(1):33--42, 2003.

\bibitem[BMPP20]{BMPP}
J.~B\"{o}ttcher, R.~Montgomery, O.~Parczyk, and Y.~Person.
\newblock Embedding spanning bounded degree graphs in randomly perturbed graphs.
\newblock {\em Mathematika}, 66(2):422--447, 2020.

\bibitem[BTW19]{BTW}
J.~Balogh, A.~Treglown, and A.~Z. Wagner.
\newblock Tilings in randomly perturbed dense graphs.
\newblock {\em Combin. Probab. Comput.}, 28(2):159--176, 2019.

\bibitem[CLWY26]{CLWY}
X.~Cheng, H.~Liu, L.~Wang, and Z.~Yan.
\newblock Triangle packings in randomly perturbed graphs.
\newblock {\em arxiv:2604.25250}, 2026.

\bibitem[DT20]{DT}
S.~Das and A.~Treglown.
\newblock Ramsey properties of randomly perturbed graphs: cliques and cycles.
\newblock {\em Combin. Probab. Comput.}, 29(6):830--867, 2020.

\bibitem[EDR25]{ER}
A.~Espuny~D\'{\i}az and R.V. Razafindravola.
\newblock How many random edges make an almost-{D}irac graph {H}amiltonian?
\newblock {\em Electron. J. Combin.}, 32(4):Paper No. 4.47, 15, 2025.

\bibitem[HKMM{\etalchar{+}}21]{HMMMP}
M.~Hahn-Klimroth, G.~S. Maesaka, Y.~Mogge, S.~Mohr, and O.~Parczyk.
\newblock Random perturbation of sparse graphs.
\newblock {\em Electron. J. Combin.}, 28(2):Paper No. 2.26, 12, 2021.

\bibitem[HMT21]{HMT}
J.~Han, P.~Morris, and A.~Treglown.
\newblock Tilings in randomly perturbed graphs: bridging the gap between {H}ajnal-{S}zemer\'{e}di and {J}ohansson-{K}ahn-{V}u.
\newblock {\em Random Structures Algorithms}, 58(3):480--516, 2021.

\bibitem[JK20]{JK}
F.~Joos and J.~Kim.
\newblock Spanning trees in randomly perturbed graphs.
\newblock {\em Random Structures Algorithms}, 56(1):169--219, 2020.

\bibitem[KKS17]{KKS}
M.~Krivelevich, M.~Kwan, and B.~Sudakov.
\newblock Bounded-degree spanning trees in randomly perturbed graphs.
\newblock {\em SIAM J. Discrete Math.}, 31(1):155--171, 2017.

\bibitem[P\'76]{posa}
L.~P\'{o}sa.
\newblock Hamiltonian circuits in random graphs.
\newblock {\em Discrete Math.}, 14(4):359--364, 1976.

\end{thebibliography}
\end{document}